\documentclass{article}
\usepackage{xspace}

\usepackage[a4paper,outer=3cm]{geometry}

\usepackage[utf8]{inputenc}
\usepackage[T1]{fontenc}
\usepackage{times}

\usepackage{xspace}

\usepackage{longtable}
\usepackage{amsmath,amsthm,amssymb,amsthm, stmaryrd}
\usepackage{enumerate}

\usepackage{algorithm}
\usepackage{algpseudocode}
\algnewcommand\algorithmicinput{\textbf{Input:}}
\algnewcommand\algorithmicoutput{\textbf{Output:}}
\algnewcommand\Input{\item[\algorithmicinput]}
\algnewcommand\Output{\item[\algorithmicoutput]}

\usepackage{hyperref}

\usepackage{subcaption}
\usepackage{pgf,tikz}
\usepackage{svg}
\usetikzlibrary{arrows}
\usetikzlibrary{decorations.markings}

\usepackage{pdftricks}
\begin{psinputs}
	\usepackage{pstricks,pst-plot,pst-node,pst-func}
\end{psinputs}
\usepackage{graphics,graphicx}

\tikzstyle{vertex}=[circle, draw, inner sep=0pt, minimum size=6pt]

\tikzset{->-/.style={decoration={
  markings,
  mark=at position .5 with {\arrow{>}}},postaction={decorate}}}

 \usepackage{marginnote}

\newcommand{\m}[1]{}

\usepackage[nothm]{thmbox}
\usepackage{amsthm}
\usepackage{thmtools}

\declaretheorem[parent=section,thmbox=M]{theorem}

\declaretheorem[numberlike=theorem]{lemma}

\newcommand{\mc}{\mathcal}

\newcommand{\ra}{\rightarrow}
\newcommand{\Ra}{\Rightarrow}

\newcommand{\ora}[1]{\overrightarrow{#1}}

\DeclareMathOperator{\dic}{\ora \chi}

\newcommand{\F}{Forb_{ind}}

\newcommand{\overbar}[1]{\mkern 1.7mu\overline{\mkern-1.7mu#1\mkern-1.7mu}\mkern 1.7mu}

\tikzstyle{vertex}=[circle,draw, top color=gray!5, 
	    bottom color=gray!30, minimum size=12pt, scale=1, inner sep=0.5pt]
\tikzstyle{arc}=[->, > = latex',  thick]
\tikzstyle{edge}=[thick, blue]

\newcommand{\ocg}{orientation of a chordal graph\xspace}
\newcommand{\ocgs}{orientations of chordal graphs\xspace}

	{\noindent {}{#1}{}}{ This proves~(\arabic{claim}).\vspace{2ex}}

\title{Heroes in orientations of chordal graphs}

\author{Pierre Aboulker$^1$, Guillaume Aubian$^{1,3}$, Raphael Steiner$^{2}$\\
\small ($1$) DIENS, \'Ecole normale sup\'erieure, CNRS, PSL University, Paris, France\\
\small ($2$) ETH Z\"{u}rich, Switzerland\\
\small ($3$) Université de Paris, CNRS, IRIF, F-75006, Paris, France}

\begin{document}

\maketitle

\begin{abstract}
We characterize all digraphs $H$ such that orientations of chordal graphs with no induced copy of $H$ have bounded dichromatic number. 
\end{abstract}

\section{Introduction}
Throughout this paper, we only consider simple digraphs $G$, that is, for every two distinct vertices $u$ and $v$, the digraph $G$ contains either an
arc from $u$ to $v$, or an arc from $v$ to $u$, or neither; but not both. 
Given a digraph $G$, we denote by $V(G)$ its set of vertices and $A(G)$ its set of arcs. 
For a vertex $x$ of a digraph $G$, we denote by $x^+(G)$ (resp. $x^-(G)$) the set of its out-neighbours (resp. in-neighbours). 
If there is no ambiguity on the digraph, we will simply use $x^+$ and $x^-$.

A {\it $k$-dicolouring} (or \emph{acylic $k$-colouring}) of a digraph $G$ 
 is a mapping $c:V(G) \rightarrow I$ using a colour set $I$ of size $k$ such that for every colour $i \in I$, its preimage $c^{-1}(i)$ induces an acyclic subdigraph of $G$. 
The {\it dichromatic number} of $G$, denoted by $\dic (G)$ and introduced by Neumann-Lara in~\cite{N82}, is the smallest integer $k$ such that  $G$ admits a $k$-dicolouring.

A \emph{tournament} is an orientation of a complete graph. 
A {\em transitive tournament} is an acyclic tournament and we denote by $TT_k$ the unique acyclic tournament on $k$ vertices. 
Given a transitive tournament $T$ on $n$ vertices $\{v_1, \dots, v_n\}$, we say that $v_1, \dots, v_n$ is the \emph{topological ordering} of $T$ if, for all $1 \le i<j \le n$, we have $v_iv_j \in A(T)$. 
Given two tournaments $H_1$ and $H_2$, we denote by $\Delta(1,H_1,H_2)$ the tournament obtained from pairwise disjoint copies of $H_1$ and $H_2$ plus a vertex $x$, and all arcs from $x$ to the copy of $H_1$, all arcs from the copy of $H_1$ to the copy of $H_2$, and all arcs from the copy of $H_2$ to $x$.  We write $\Delta(1,k,H)$ for $\Delta(1, TT_k, H)$. For tournaments $H_1$ and $H_2$, we denote by $H_1 \Ra H_2$ the digraph obtained from disjoint copies of $H_1$ and $H_2$ by adding all  arcs from the copy of $H_1$ to the copy of $H_2$. 

Given two digraphs $G$ and $H$, we say that $G$ is \emph{$H$-free} if it does not contain an induced copy of $H$. 
Given a hereditary class of digraphs $\mc C$, we say that a digraph $H$ is a \emph{hero in $\mc C$} if every $H$-free digraph of $\mc C$ has bounded dichromatic number. 

In a breakthrough paper, Berger, Choromansky, Chudnovsky, Fox, Loebl, Scott, Seymour and Thomassé gave a recursive characterization of all heroes in tournaments, as follows. 

\begin{theorem}[Berger et al. \cite{hero}]\label{thm:heroes}
A digraph $H$ is a hero in tournaments if and only if :
\begin{itemize}
\item $H=K_1$ (the one-vertex digraph), or
\item $H=(H_1 \Ra H_2)$, where $H_1$ and $H_2$ are heroes in tournaments, or
\item $H=\Delta(1,k,H')$ or $H=\Delta(1,H',k)$, where $k\geq 1$ and $H'$ is a hero in tournaments.
\end{itemize}
\end{theorem}

Observe that if a class of digraphs $\mc C$ contains all tournaments, then a hero in $\mathcal C$ must be a hero in tournaments.  
A \emph{chordal graph} is a graph with no induced cycle of length at least $4$. A classical theorem of Dirac~\cite{D61} states that all chordal graphs can be obtained by iteratively gluing some complete graphs along cliques (see Section~\ref{sec:chordal} for a formal statement). 
 This implies for undirected graph colouring that chordal graphs are perfect graphs, and thus their chromatic numbers and colouring properties are determined solely by the (largest) cliques contained in them.
 It is then natural to ask whether also for the dichromatic number of \ocgs important characteristics are determined by the largest dichromatic numbers of their subtournaments. In particular, it is a natural problem to characterise the heroes in \ocgs and to see whether they are the same as for tournaments.
 
 In this paper, we find surprising answers to the above questions. First, there is  very few heroes in \ocgs and as our main contribution, we completely describe these digraphs, as follows. 

\begin{theorem}\label{thm:main}
A digraph $H$ is a hero in \ocgs if and only if $H$ is a transitive tournament or isomorphic to $\Delta(1,1,k)$ for some integer $k \ge 1$.
\end{theorem} 

Secondly, our constructions in the proof of the above characterisation exhibit \ocgs with arbitrarily large dichromatic number all whose subtournaments are $2$-colourable, showing that in contrast to chromatic number the dichromatic number of an \ocg heavily depends on its global structure and not only on the cliques contained in it.

We denote by $\vec C_3$ the directed cycle on three vertices, also called \emph{directed triangle} (observe that $\vec C_3 = \Delta(1,1,1)$). 
It is easy to see that a hero in tournaments is either a transitive tournament, or isomorphic to  $\Delta(1,1,k)$ for some integer $k \ge 1$, or it contains one of the heros $\Delta(1,2,2)$, $K_1 \Ra \vec C_3$ or $\vec C_3 \Ra K_1$ as a subtournament. Moreover, since  reversing all arcs of a $(\vec C_3 \Ra K_1)$-free \ocg results in a $ (K_1 \Ra \vec C_3)$-free \ocg and does not change the dichromatic number, proving that $\vec C_3 \Ra K_1$ is not a hero in \ocgs implies that $K_1 \Ra \vec C_3$ is not either. 
Hence, to prove Theorem~\ref{thm:main}, it will be enough to prove the following: 
\begin{itemize}
    \item Transitive tournaments and $\Delta(1,1,k)$ for $k \ge 1$ are heroes in \ocgs. This is done in Section~\ref{sec:heroes}.
    \item  $\Delta(1,2,2)$ and $\vec C_3 \Ra K_1$ are not heroes in \ocgs. This is respectively done in subsections~\ref{subsec:122} and~\ref{subsec:1C3}. 
\end{itemize}

\medskip

\textbf{Related results}: 

Given a digraph $H$, denote by $\F(H)$ the class of digraphs with no induced copy of $H$. 
A result of~\cite{HM12} implies that if $H$ is not an orientation of a forest, then no digraph is a hero in $\F(H)$ except for the isolated vertex and the arc. 
A systematic study of heroes in classes of digraphs of the form $\F(H)$ where $H$ is an oriented forest has been initiated in~\cite{ACN21}. 
An \emph{oriented star} is an orientation of a star, that is a tree with only one non-leaf vertex. It is proved~\cite{ACN21} that if $H$ is not the disjoint union of oriented stars, then no hero in $\F(H)$ contains a directed triangle. 
A  result in~\cite{CS19} implies that every transitive tournament is a hero in $\F(H)$ when $H$ is a disjoint union of oriented stars. 
It is proved in~\cite{HLNT19} that heroes in $\F(\overbar{K_k})$ are the same as heroes in tournaments, where $\overbar{K_k}$ is the graph on $k$ vertices with no arc (which is in particular the simplest union of disjoint oriented stars).  
In~\cite{AAC21} and~\cite{S21}, it is proved that $K_1 \Ra \vec C_3$ is a hero in $\F(\vec K_{1,2})$, where $\vec K_{1,2}$ is the star on three vertices with a vertex of out-degree $2$.  
In~\cite{AAC22multi}, heroes in the class of orientations of complete multipartite graphs (which corresponds to the class $\F(K_1 + \vec K_2)$ where $K_1 + \vec K_2$ is the graph made of an isolated vertex and an arc) are almost fully characterized, up to one particular digraph, namely $\Delta(1,2,2)$. 

The \emph{clique number} $\omega(G)$ of a digraph $G$ is the size of a largest clique in the underlying graph of $G$.  We say that a hereditary class of digraphs $\mathcal C$ is \emph{$\dic$-bounded} if there exists a function $f$ such that for every $G \in \mathcal C$, $\dic(G) \leq f(\omega(G))$. It is easy to see that a class of digraphs $\mathcal C$ is $\dic$-bounded if and only if $TT_k$ is a hero in $\mathcal C$ for every integer $k$ (this is because every orientation of a large enough complete graph contains a copy of $TT_k$). 
We denote by $\vec P_k$ the directed path on $k$ vertices.  It is proved in~\cite{CHMS22} that, for every $k \geq 3$, the class of digraphs with no induced  $\vec P_k$ and no induced directed cycle of length at most $k-1$ is $\dic$-bounded.  
It is proved in~\cite{ABV22} that the class of digraphs with no induced directed cycle of length at least $4$ is not $\dic$-bounded (more precisely it is proved that $TT_3$ is not a hero in the class).

\section{Proofs}

\subsection{A few words on chordal graphs}\label{sec:chordal}

A graph $G$ is \emph{chordal} if it contains no induced cycle of length at least $4$. Chordal graphs have been studied for the first time in the pioneer work of Dirac~\cite{D61} who proved that every chordal graph $G$ is either a complete graph, or contains a clique $S$ such that $G \setminus S$ is disconnected. This easily implies that all chordal graphs can be obtained by gluing complete graphs along cliques. From this point of view, it is natural to try to generalize results on tournaments to orientations of chordal graphs. 

In this paper, we will use the two following well-known properties of chordal graphs. The first one formalizes the notion of `gluing along a clique'.

\begin{lemma}\cite{D61}\label{lem:chordal_intersection_on_clique}
Let $G_1$ and $G_2$ be two chordal graphs such that $V(G_1) \cap V(G_2)$ induces a complete graph both in $G_1$ and $G_2$. Then their union is a chordal graph.
\end{lemma}

A vertex is \emph{simplicial} if its neighborhood induces a complete graph. 
\begin{lemma}\cite{D61}\label{lem:simplicial}
Every chordal graph has a simplicial vertex. 
\end{lemma}

\subsection{$\Delta(1,1,k)$ and transitive tournaments are heroes in \ocgs}\label{sec:heroes}

\begin{theorem}[Stearns, \cite{stearns}]\label{thm:stearns}
For each integer $n \geq 1$, a tournament with at least $2^{n-1}$ vertices contains a transitive tournament with $n$ vertices.
\end{theorem}

In the following, we define the \emph{triangle degree} of a vertex $x$ in a digraph $G$ as the maximum size of a collection of directed triangles that pairwise share the common vertex $x$ but no further vertices.

\begin{lemma}\label{lem:bounded_triangle_degree}
Every vertex of a $\Delta(1,1,k)$-free tournament has triangle degree less than $2^{2k-2}$. 
\end{lemma}

\begin{proof}
Let $G$ be a $\Delta(1,1,k)$-free tournament and $x$ a vertex of $G$. 
Assume for contradiction that $x$ has triangle degree at least $2^{2k-2}$, that is, there exist pairwise distinct vertices $a_1, b_1, \dots, a_{2^{2k-2}}, b_{2^{2k-2}}$ such that $x \ra a_i \ra b_i \ra x$. 
By Theorem~\ref{thm:stearns} we can find a transitive tournament $T$ in $G[\{a_1, \dots, a_{2^{2k-2}}\}]$ of size at least $2k-1$. Up to renaming the vertices, we may assume that $T=G[\{a_1, \dots, a_{2k-1}\}]$ and that $a_1, \dots, a_{2k-1}$ is the topological ordering of $T$.   
Then look at $b_{2k-1}$. Set $b_{2k-1}^+ \cap T  = T^+$ and $b_{2k-1}^- \cap T = T^-$ and observe that $V(T) = T^+ \cup T^-$ since we are in a tournament. 
If $|T^+| \geq k$, then $T^+$ together with $b_{2k-1}$ and $a_{2k-1}$ contains a  $\Delta(1,1,k)$, a contradiction. So $|T^+|\leq k-1$. 
If $|T^-| \geq k$, then $T^-$ together with $b_{2k-1}$ and $x$ contains $\Delta(1,1,k)$, a contradiction. So $|T^+|\leq k-1$. Hence, $|V(T)| \leq 2k-2$, a contradiction. 
\end{proof}

\begin{theorem}
Transitive tournaments and $\Delta(1,1,k)$ are heroes in \ocgs. 
More precisely, $TT_k$-free \ocgs have dichromatic number at most $2^{k-1}-1$ and  $\Delta(1,1,k)$-free \ocgs have dichromatic number at most $2^{2k-2}$.  
\end{theorem}

\begin{proof}
A $TT_k$-free \ocg has no clique of size at least $2^{k-1}-1$ by Theorem~\ref{thm:stearns}, and since chordal graphs are perfect graphs, its underlying graph has chromatic number at most $2^{k-1}-1$ and thus dichromatic number at most $2^{k-1}-1$.

We now prove that $\Delta(1,1,k)$-free \ocgs have dichromatic number at most $2^{2k-2}$.  We proceed by induction on the number of vertices. 
Let $G$ be a $\Delta(1,1,k)$-free \ocg. Let $x$ be a simplicial vertex of the underlying graph of $G$.  Note that the triangle degree of $x$ in $G$ is equal to the triangle degree of $x$ in the subtournament $G[\{x\} \cup x^+ \cup x^-]$, which by Lemma~\ref{lem:bounded_triangle_degree} is less than $2^{2k-2}$.

We can then find an acyclic colouring of $G \setminus x$ with $2^{2k-2}$ colours by induction,  and since the triangle degree of $x$ in $G$ is less than $2^{2k-2}$, there is a colour $i \in \{1, \dots, 2^{2k-2}\}$ such that assigning $i$ to $x$ does not produce a monochromatic directed triangle. The resulting colouring is thus an acyclic colouring of $G$: For if there existed a monochromatic directed cycle in this colouring of $G$, there would also have to exist an \emph{induced} monochromatic directed cycle, and since all induced cycles in $G$ have length $3$, this cycle would have to be a monochromatic directed triangle. However, such a triangle does not exist, neither through $x$ nor in $G\setminus x$ (by inductive assumption). 
\end{proof}

\subsection{Constructions} \label{sec:nonheroes}

\subsubsection{$\Delta(1,2,2)$ is not a hero in orientations of chordal graphs}\label{subsec:122}

In this subsection, we present a construction of \ocgs with arbitrarily large dichromatic number but containing no copy of $\Delta(1,2,2)$. 

\begin{theorem}\label{thm:ce_122}
$\Delta(1,2,2)$ is not a hero in \ocgs.
\end{theorem}

\begin{proof}
We inductively construct a sequence $(G_k)_{k \in \mathbb{N}}$ of digraphs such that for each $k \ge 1$, the digraph $G_k$ is an orientation of a chordal graph with no copy of $\Delta(1,2,2)$ satisfying $\vec{\chi}(G_k)=k$.

Let $G_1$ be the digraph on one vertex, and having defined $G_{k}$, define $G_{k+1}$ as follows. 
Start with a copy $T$ of $TT_{k+1}$, and for each arc $e=uv$ of $T$, create a distinct copy $G_{k}^e$ of $G_{k}$ (vertex-disjoint for different choices of the arc $e \in A(T)$, and all vertex-disjoint from $T$). Next, for each $e=uv \in A(T)$, we add all the arcs $vy$ and $yu$ for every $y \in V(G_k^e)$.  
This completes the description of the digraph $G_{k+1}$.

For every arc $e=uv \in A(T)$, consider the underlying graph of $G_{k+1}[\{u,v\} \cup V(G_k^e)]$. By definition, this graph is obtained from the chordal underlying graph of $G_k^e$ by adding an adjacent pair of universal vertices. Since the addition of universal vertices preservers the chordality of a graph, we can see that the underlying graph of $G_{k+1}[\{u,v\} \cup V(G_k^e)]$ is chordal, for every choice of $e$. Since $T$ and $G_{k+1}[\{u,v\} \cup V(G_k^e)]$ intersect in  the clique $\{u,v\}$, we may now repeatedly apply Lemma~\ref{lem:chordal_intersection_on_clique} to see that $G_{k+1}$ is still an \ocg.

Next, let us prove that $G_{k+1}$ does not contain $\Delta(1,2,2)$. Assume towards a contradiction that $G_{k+1}$ contains a copy of $\Delta(1,2,2)$, induced by the set of vertices $A \subseteq V(G_{k+1})$. Since the copies $G_{k}^e, e \in A(T)$ of $G_k$ are vertex disjoint and have no connecting arcs, and since $A$ induces a tournament, $A$ intersects at most one of the vertex sets of these copies. Let $f=xy \in A(T)$ be a fixed edge such that $A \subseteq V(T) \cup V(G_k^f)$. 

Since $G_k^f$ is $\Delta(1,2,2)$-free by inductive assumption, it follows that $A$ intersects $V(T)$ in at least one vertex. 
As $\Delta(1,2,2)$ is not acyclic, $A$ is also not fully contained in $V(T)$, and thus $A \cap V(G_k^f) \neq \emptyset$. 

The argument above implies that $A \cap V(T) \subseteq \{x,y\}$, as $x$ and $y$ are the only vertices in $V(T)$ whose neighborhoods in $G_{k+1}$ intersect $V(G_k^f)$. 
In fact, we must have $A \cap V(T)=\{x,y\}$, for if $|A\cap V(T)|=1$ then either $x$ would form a sink in $G_{k+1}[A]$ or $y$ would form a source in $G_{k+1}[A]$, both of which are impossible, since $G_{k+1}[A] \simeq \Delta(1,2,2)$ is strongly connected. Note that by definition of $G_{k+1}$, every vertex in $A \setminus \{x,y\} \subseteq V(G_k^f)$ must form a directed triangle together with the arc $xy$. 

But $A$ induces $\Delta(1,2,2)$ in $G_{k}$ and there is no arc in $\Delta(1,2,2)$ forming a directed triangle with every other vertex, as there is no arc from the only vertex of $\Delta(1,2,2)$ of outdegree $1$ to the only vertex of $\Delta(1,2,2)$ of indegree $1$, a contradiction.  This shows that $G_{k+1}$ is indeed $\Delta(1,2,2)$-free.

Finally, let us prove that $\dic(G_k)=k+1$. A $(k+1)$-dicolouring of $G_k$ can easily be obtained by piecing together individual $k$-dicolourings of the copies $G_k^e, e \in A(T)$ of $G_k$ and assigning to all vertices in the transitive tournament $T$ a new $(k+1)^{th}$ colour not appearing in the copies. To show that $\dic(G_{k+1})>k$, assume towards a contradiction that $G_k$ admits a $k$-dicolouring $c:V(G_{k+1}) \rightarrow \{1,\ldots,k\}$. 
Then, since $T$ is a clique on $k+1$ vertices, there exists a monochromatic arc $e=uv$. Let $i \in \{1,\ldots,k\}$ be such that $c(u)=c(v)=i$. Then since $\dic(G_k)=k$, the copy $G_k^e$ of $G_k$ glued to $uv$ must use all $k$ colours in the dicolouring induced on it by $c$, and in particular there exists some $w \in V(G_k^e)$ such that $c(w)=i$. Now, however, the directed triangle $x \rightarrow y \rightarrow w \rightarrow x$ is monochromatic, a contradiction to our choice of $c$. This completes the proof that $\dic(G_{k+1})=k+1$, and hence the proof of the theorem. 
\end{proof}

\subsubsection{$\vec C_3 \Ra K_1$ is not a hero in orientations of chordal graphs}\label{subsec:1C3}

All along this subsection, we denote by $\mc C$ the class of $(\vec C_3 \Ra K_1)$-free \ocgs. The goal of this subsection is to construct digraphs in $\mc C$ with arbitrarily large dichromatic number.

\begin{lemma}\label{lem:directedsum}
Let $G, F \in \mc C$ and let $T$ be a transitive subtournament of $G$. Then the digraph $K$ obtained from $G$ and $F$ by adding every arc from $T$ to $F$ is in $\mc C$. 
\end{lemma}

\begin{proof}
 Given a graph $G$, the graph obtained by adding a vertex $v$ adjacent with every vertex of $G$ results in a chordal graph  as, if $v$ lies in an induced cycle, there is an arc between $v$ and every other vertex of this cycle, which is thus a triangle.
 Thus, adding vertices of $T$ to $F$ one by one, together with all arcs from $T$ to $F$, returns a chordal graph $F'$. The intersection of $V(F')$ and  $V(G)$ is $T$, which is a tournament. Hence, by Lemma~\ref{lem:chordal_intersection_on_clique}, the union of $G$ and $F'$, that is $K$, is an \ocg.

Suppose for contradiction that $K$ contains a subgraph $H$ isomorphic to  $\vec C_3 \Ra K_1$.
Since $G,F \in \mathcal C$, $H$ must intersect both $G$ and $F$ and since $H$ is a tournament, it must be included in $T \cup F$. 
Since there is no arc from $F$ to $T$, the directed triangle of $H$ cannot intersect both $T$ and $F$, and hence must be included in $F$ (as $T$ is a transitive tournament and thus have no directed triangle). The fourth vertex of $H$ contains the directed triangle in its in-neighborhood, and thus must also be in $F$, a contradiction.  
\end{proof}

\begin{lemma}\label{lem:add_vertex_TT}
Let $G \in \mc C$ and let $T$ be a transitive subtournament of $G$ on vertices $\{v_1, \dots, v_n\}$ such that $v_1, \dots, v_n$  is the topological ordering of $T$. 
Then for every $j \in \{1, \dots, n-1\}$, the digraph $F$ obtained from $G$ by adding a vertex $x$ that sees $v_1, \dots, v_j$ and is seen by $v_{j+1}, \dots, v_n$ is in $\mc C$. 
\end{lemma}

\begin{proof}
By Lemma~\ref{lem:chordal_intersection_on_clique}, $F$ is an \ocg. Assume for contradiction that $F$ contains a copy $H$ of $K_1 \Ra \vec C_3$. Since $G \in \mc C$, $H$ must contain $x$ and thus be included in $G[K]$ where $K=V(T) \cup \{x\}$. Now, observe that $x^- \cap K$, $v_i^- \cap K$ for $i=1, \dots, j$ and $v_k^- \cap K$ for $k=j+1, \dots, n$ are transitive tournaments. Thus $G[K]$ cannot contain $H$, since one vertex in $H$ includes a directed triangle in its in-neighbourhood.  
\end{proof}

In the following, given a $k$-colouring $c:V(F) \rightarrow \{1,\ldots,k\}$ of a digraph $F$, we say that a subdigraph of $F$ is \emph{rainbow} (with respect to $c$), if its vertices are assigned pairwise distinct colours. 

\begin{lemma}\label{lem:rainbow}
Let $G \in \mc C$ such that $\dic(G)=k$. There exists a digraph $F=F(G) \in \mc C$ with $\dic(F)=k$ satisfying the following property: For every $k$-dicolouring of $F$, there exists a rainbow transitive tournament of size $k$ contained in $F$.
\end{lemma}
\begin{proof}
We prove the lemma by showing the following statement using induction on $i$ (the lemma then follows by setting $F:=F^{(k)}$). 

\medskip

$(\star)$ For every $i \in \{1,\ldots,k\}$, there exists a digraph $F^{(i)} \in \mc C$ such that $\dic(F^{(i)})=k$, and for every $k$-dicolouring of $F^{(i)}$, there exists a copy of $TT_i$ contained in $F^{(i)}$ which is rainbow.

\medskip

The statement of $(\star)$ is trivially true for $i=1$, since we may put $F^{(1)}:=G$, and in every $k$-dicolouring of $F^{(1)}$ any single vertex forms a rainbow $TT_1$. 

For the inductive step, let $i \in \{1,\ldots,k-1\}$ and suppose we have established the existence of a digraph $F^{(i)} \in \mc C$ of dichromatic number $k$ such that every $k$-dicolouring of $F^{(i)}$ contains a rainbow copy of $TT_i$. 

We now construct a digraph $F^{(i+1)}$ from $F^{(i)}$ as follows: Let $\mathcal{X}$ denote the set of all $X \subseteq F^{(i)}$ such that $X$ induces a $TT_i$ in $F^{(i)}$. Now, for every $X \in \mathcal{X}$ create a distinct copy $G_X$ of the digraph $G$ (pairwise vertex-disjoint for different choices of $X$, and all vertex-disjoint from $F^{(i)}$). Finally, for every $X\in \mathcal{X}$, add all the arcs $xy$ with $x \in X$ and $y \in V(G_X)$. Since $F^{(i)} \in \mc C$ and $G_X \in \mc C$ for every $X \in \mathcal{X}$, we can repeatedly apply Lemma~\ref{lem:directedsum} to find that the resulting digraph, which we call $F^{(i+1)}$, is still contained in $\mc C$. 

Note that by construction, no directed cycle in $F^{(i+1)}$ intersects more than one of the vertex-disjoint subdigraphs $F^{(i)}$ and  $(G_X|X \in \mathcal{X})$ of $F^{(i+1)}$, and hence, these digraphs may be coloured independently in every dicolouring of $F^{(i+1)}$. This immediately implies $\dic(F^{(i+1)})=\max\{\dic(F^{(i)}),\dic(G)\}=k$. 

To prove the inductive claim, consider any $k$-dicolouring $c:V(F^{(i+1)}) \rightarrow \{1,\ldots,k\}$ of $F^{(i+1)}$. Then by inductive assumption, there exists a rainbow copy of $TT_i$ contained in the subdigraph of $F^{(i+1)}$ isomorphic to $F^{(i)}$. Let $X$ denote its vertex-set, and let $I \subseteq \{1,\ldots,k\}$ be the set of $i$ distinct colours used on $X$. Since $i<k$ and $\dic(G_X)=k$, there exists a vertex $v \in V(G_X)$ such that $c(v) \notin I$. Now, the vertex-set $X \cup \{v\}$ induces a rainbow $TT_{i+1}$ contained in $F^{(i+1)}$, as desired. This proves $(\star)$ and thus the lemma. 
\end{proof}

\begin{theorem}
The digraph $\vec C_3 \Ra K_1$ is not a hero in \ocgs. 
\end{theorem}
\begin{proof}
We construct a sequence of digraphs $(G_k)_{k \in \mathbb N}$ such that $\dic(G_k)=k$ and $G_k \in \mc C$.  
Let $G_1$ be the one-vertex-digraph and, having defined $G_k$, define $G_{k+1}$ as follows. Let $F_k:=F(G_k) \in \mc C$ be the digraph given by Lemma~\ref{lem:rainbow}, such that $\dic(F_k)=k$ and such that every $k$-dicolouring of $F_k$ contains a rainbow copy of $TT_k$. 

Let $\mathcal{T}$ denote the set of transitive tournaments which are subdigraphs of $F_k$. Now, for each transitive subtournament $T \in \mathcal{T}$, add a copy $F^T_k$ of $F_k$ (vertex-disjoint for different choices of $T$, and all vertex-disjoint from $F_k$). Next, for every $T \in \mathcal{T}$, add all the arcs $xy$ with $x \in V(T)$ and $y \in V(F^T_k)$. Finally, for every choice of $T \in \mathcal{T}$ and every transitive subtournament $T'$ of $F^T_k$, add a vertex $x_{T,T'}$ that is seen by every vertex of $T'$ and that sees every vertex of $T$.  
This completes the description of the digraph $G_{k+1}$.

By repeatedly applying Lemma~\ref{lem:directedsum} and Lemma~\ref{lem:rainbow}, we can see that all of the operations performed to construct $G_{k+1}$ preserve containment in $\mc C$, and hence, since $F_k \in \mc C$, we also must have $G_{k+1} \in \mc C$. 

Let us now prove that $\dic(G_{k+1}) = k+1$. A $(k+1)$-dicolouring can be achieved by piecing together individual $k$-dicolourings of $F_k$ and its copies $F_k^T, T \in \mathcal{T}$, and assigning to all vertices of the form $x_{T,T'}$ (which form a stable set in $G_{k+1}$) a distinct $(k+1)$-th colour. 

Finally, to prove that $\dic(G_{k+1})>k$, assume towards a contradiction that $G_{k+1}$ admits a dicolouring using colours from $\{1,\ldots,k\}$. Then by Lemma~\ref{lem:rainbow}, in this colouring $F_k$ contains a rainbow transitive tournament $T$ of size~$k$. Again by Lemma~\ref{lem:rainbow}, also $F^T_k$ contains a rainbow transitive subtournament $T'$ of size $k$. Now consider the vertex $x_{T,T'}$ in $G$, and let $i \in \{1,\ldots,k\}$ denote its colour. Since both $T$ and $T'$ contain all $k$ colours, there exist vertices $t_1 \in V(T)$ and $t_2 \in V(T')$ which are both assigned colour $i$. Finally, this yields a contradiction, since now the directed triangle $t_1 \rightarrow t_2 \rightarrow x_{T,T'} \rightarrow t_1$ in $G_{k+1}$ is monochromatic. 
\end{proof}

\section{Further works}

After characterising heroes in \ocgs, it is natural to ask what are the heroes in orientations of subclasses or superclasses of chordal graphs. 

Concerning superclasses of chordal graphs, consider the following construction (already mentioned in~\cite{ACN21}). Let $G_1$ be the graph on $1$ vertex, and having defined $G_{k-1}$ inductively, define $G_k$ as follows: start with three disjoint copies $G^1_{k-1}, G^2_{k-1}, G^3_{k-1}$ of $G_{k-1}$ plus a vertex $x$, and add all arcs from $x$ to $V(G^1_{k-1})$, all arcs from $V(G^1_{k-1})$ to $V(G^2_{k-1})$, all arcs from $V(G^2_{k-1})$ to $V(G^3_{k-1})$ and finally, all arcs from $V(G^3_{k-1})$ to $x$. It is then easy to see that $\dic(G_k)=k$ and that the underlying graph of $G_k$ does not contain induced path of length $4$. Hence, the underlying graphs of the $G_k$'s are perfect graphs, and even co-graphs, which implies that $\vec C_3$ is not a hero in orientation of perfect graphs. So the only possible heroes are transitive tournaments, which are trivially, since transitive tournaments are heroes in any orientations of graphs in $\mc C$, whenever $\mc C$ is a $\chi$-bounded class of graphs.

Concerning subclasses of chordal graphs, orientations of interval graphs seems to be an intriguing case. On one hand, we were not able to decide whether or not $\Delta(1,2,2)$ or $\vec C_3 \Ra K_1$ are heroes in this class,  and our attempts have not led us to a strong opinion as to the answer. 
On the other hand, we can prove the following. A \emph{unit interval graph} is an interval graph that admits an interval representation in which every interval has unit length. 

\begin{theorem}
Heroes in orientations of unit interval graphs are the same as heroes in tournaments.
\end{theorem}

\begin{proof}
Since complete graphs are unit interval graphs, the set of heroes in  orientations of proper interval graphs is a subset of the set of heroes in tournaments. 

We are going to prove the following, which easily implies that every hero in tournaments is a hero in orientation of unit interval graphs. \smallskip

$(\star)$ For every integer $C$, if $G$ is an orientation of a unit interval graph in which every subtournament has dichromatic number at most $C$, then $G$ is $2C$-dicolourable. 
\smallskip

Let $G$ be an orientation of a unit interval graph and $C$ an integer such that every subtournament of $G$ has dichromatic number at most $C$. 
Consider an interval representation of $G$ where each interval has length $1$ and assume without loss of generality that the endpoints of each interval are not integers. 
 For every integer $k$, let $K_k$ be the set of vertices of $G$ whose associated interval contains $k$. So each $K_k$ induces a subtournament of $G$, and by hypothesis, $G[K_k]$ is $C$-dicolourable. Moreover, since each interval has length $1$ and their extremities are not integers, the $K_k$'s  partition the vertices of $G$ and  there is no arc between $K_i$ and $K_j$ whenever $|i-j| \geq 2$.  Hence, piecing together dicolourings of $G[K_k]$ with colours from $\{1, \dots, C\}$ when $k$ is odd, and from $\{C+1, \dots, 2C\}$ when $k$ is even, results in a $2C$-dicolouring of $G$. 
\end{proof}

We say that a digraph is \emph{$t$-local} if the out-neighborhood of each of its vertices induces a digraph with dichromatic number at most $t$. A class of digraphs $\mc C$ has the \emph{local to global property} if, for every integer $t$, $t$-local digraphs in $\mc C$ have bounded dichromatic number. 
It is proved in~\cite{HLTW19} that tournaments have the local to global property, and this result was generalised to the class of digraphs with bounded independence number in~\cite{HLNT19}.  
Since $K_1 \Ra  \vec C_3$ is not a hero in \ocgs, we get that the class of \ocgs does not have the local to global property, and that even $1$-local \ocgs can have arbitrarily large dichromatic number. 
We wonder if other interesting classes of digraphs have it. . 
\bigskip

{\bf Acknowledgments:} 
This research was partially supported by ANR project DAGDigDec (JCJC)   ANR-21-CE48-0012 and by the group Casino/ENS Chair on Algorithmics and Machine Learning. 
Raphael Steiner was supported by an ETH Zurich Postdoctoral Fellowship.

\end{document}